\documentclass{article}
\usepackage{amsmath}
\usepackage{fancyhdr}

\newtheorem{theorem}{Theorem}

\newtheorem{corollary}[theorem]{Corollary}

\newtheorem{definition}[theorem]{Definition}

\newtheorem{lemma}[theorem]{Lemma}

\newtheorem{proposition}[theorem]{Proposition}
\newtheorem{remark}[theorem]{Remark}

\newenvironment{proof}[1][Proof]{\noindent\textbf{#1.} }{\ \rule{0.5em}{0.5em}}
\input{tcilatex}
\input{tcilatex}
\begin{document}

\date{\today}
\title{On a class of submanifolds in tangent bundle with $g-$ natural metric}
\author{Stanis\l aw Ewert-Krzemieniewski \\
%EndAName
West Pomeranian University of Technology Szczecin\\
School of Mathematics\\
Al. Piast\'{o}w 17\\
70-310 Szczecin \\
Poland\\
e-mail: ewert@zut.edu.pl}
\maketitle

\begin{abstract}
An isometric immersion of a Riemannian manifold $M$ into a Riemannian
manifold $N$ gives rise in a natural way to the immersion of the tangent
bundle $TM$ into the tangent bundle $TN$ with a non-degenerate $g-$ natural
metric $G.$


\end{abstract}

\section{Introduction}

Let $\pi :TN\longrightarrow N$ be the tangent bundle of a manifold $N$ with
Levi-Civita connection $\nabla $ on $N,$ $\pi $ being projection. Then at
each point $(x,u)\in TN$ the tangent space $T_{(x,u)}TN$ splits into direct
sum of two isomorphic spaces $V_{(x,u)}TN$ and $H_{(x,u)}TN,$ where 
\begin{equation*}
V_{(x,u)}TN=Ker(d\pi |_{(x,u)}),\quad H_{(x,u)}TN=Ker(K|_{(x,u)})
\end{equation*}%
and $K$ is called the connection map (\cite{Dombr}) see also (\cite{1}).

More precisely, if $Z=\left( Z^{r}\frac{\partial }{\partial x^{r}}+\overline{%
Z}^{r}\frac{\partial }{\partial u^{r}}\right) |_{(x,u)}\in T_{(x,u)}TN,$
then the vertical and horizontal projections of $Z$ on $T_{x}N$ are given by%
\begin{equation*}
\left( d\pi \right) _{(x,u)}Z=\overline{Z}^{r}\frac{\partial }{\partial x^{r}%
}|_{x},\quad K_{(x,u)}(Z)=\left( \overline{Z}^{r}+u^{s}Z^{t}\Gamma
_{st}^{r}\right) \frac{\partial }{\partial x^{r}}|_{x}.
\end{equation*}

On the other hand, to each vector field $X$ on $N$ there correspond uniquely
determined vector fields $X^{v}$ and $X^{h}$\ on $TN$ such that 
\begin{eqnarray*}
d\pi |_{(x,u)}(X^{v}) &=&0,\quad K|_{(x,u)}(X^{v})=X, \\
K|_{(x,u)}(X^{h}) &=&0,\quad d\pi |_{(x,u)}(X^{h})=X.
\end{eqnarray*}%
$X^{v}$ and $X^{h}$ are called the vertical lift and the horizontal lift of
a given $X$ to $TN$ respectively.

In local coordinates $(x^{r},u^{r})$ on $TN,$ the horizontal and vertical
lifts of a vector field $X=X^{r}\frac{\partial }{\partial x^{r}}$ on $N$\ to 
$TN$ are vector fields given respectively by%
\begin{equation*}
X^{h}=X^{r}\frac{\partial }{\partial x^{r}}-u^{s}X^{t}\Gamma _{st}^{r}\frac{%
\partial }{\partial u^{r}},\quad X^{v}=X^{r}\frac{\partial }{\partial u^{r}}.
\end{equation*}

In the paper we shall frequently use the frame $(\partial _{k}^{h},\partial
_{l}^{v})=\left( \left( \frac{\partial }{\partial x^{k}}\right) ^{h},\left( 
\frac{\partial }{\partial x^{l}}\right) ^{v}\right) $ known as the adapted
frame.

Having given an isometric immersion $f:M\longrightarrow N,$ we have two
tangent bundles $\pi _{N}:TN\longrightarrow N$ and $\pi
_{M}:TM\longrightarrow M,$ where the latter is the subbundle of the former
one. Let $M,$ $N$ be two Riemannian manifolds with metrics $g_{M}$ and $%
g_{N} $ and Levi-Civita connections $\nabla _{M}$ and $\nabla _{N}$
respectively. Then $T_{p}TM$ and $T_{p}TN$ have at a common point $p$ their
own decompositions into vertical and horizontal parts, ie.%
\begin{equation*}
T_{p}TM=V_{p}TM\oplus H_{p}TM=V_{M}\oplus H_{M}
\end{equation*}%
and%
\begin{equation*}
T_{p}TN=V_{p}TN\oplus H_{p}TN=V_{N}\oplus H_{N},
\end{equation*}%
but neither $V_{M}\subset V_{N}$ nor $H_{M}\subset H_{N}$ need to hold along 
$TM.$

So, for a vector $X$ tangent to $M$ we define two vertical lifts $X^{v_{M}},$
$X^{v_{N}}$ and two horizontal lifts $X^{h_{M}},$ $X^{h_{N}}$ with respect
to the bundles over $M$ and $N$ respectively and find relations between
them. These allow us to compute the shape operator of the immersion under
consideration (see (\ref{Lift 2}) below) and make some conclusions about
just obtained submanifold of $TN$.

Notice that totally geodesic submanifolds of tangent bundle with $g-$
natural metric\ are also studied in (\cite{Abb Yamp}) and (\cite{Ewert 3}).

Throughout the paper all manifolds under consideration are Hausdorff and
smooth ones. The metrics on the base manifolds are Riemannian ones and the
metrics on tangent spaces are non-degenerat.

\section{Preliminaries on submanifolds}

\QTP{Body Math}
Throughout the paper we assume that indices $h,i,j,k,l,r,s,t$ run through
the range $1,...,n,$ while $a,b,c,d,e$ run through the range $1,...,m,$ and $%
m<n.$ Moreover, the indices $x,y,z=m+1,...,n.$

\QTP{Body Math}
Let $(N,g),$ dim$N=n,$ be a Riemannian manifold with metric $g,$ covered by
coordinate neighbourhoods $(U,$ $(x^{j})),$ $j=1,...,n.$ Let $(M,\widetilde{g%
})$ be a Riemannian manifold covered by coordinate neighbourhoods $%
(V,(y^{a})),$ $a=1,...,m,$ isometrically immersed in $(N,g)$ and let the
local expression for this immersion be $x^{r}=x^{r}(y^{a}),$ $r=1,...,n,$ $%
a=1,...,m.$ Put $\partial _{r}=\frac{\partial }{\partial x^{r}}$ and $%
B_{a}^{r}=\frac{\partial x^{r}}{\partial y^{a}}.$

\QTP{Body Math}
For the local immersion $x^{r}=x^{r}(y^{a})$ the components of the
Levi-Civita connection $\nabla $ of the induced metric $g_{ab}=g(B_{a}^{r}%
\partial _{r},B_{b}^{s}\partial _{s})=g_{rs}B_{a}^{r}B_{b}^{s}$ are%
\begin{equation*}
\Gamma _{ab}^{c}=\left[ B_{a.b}^{r}+\Gamma _{st}^{r}B_{a}^{s}B_{b}^{t}\right]
B_{r}^{c},\quad B_{r}^{c}=g^{cd}B_{d}^{t}g_{tr}.
\end{equation*}%
The van der Waerden-Bertolotti covariant derivative of $B_{a}^{r}$ is
defined by%
\begin{equation}
\nabla _{b}B_{a}^{r}=B_{a.b}^{r}+\Gamma _{st}^{r}B_{a}^{s}B_{b}^{t}-\Gamma
_{ab}^{c}B_{c}^{r},  \label{Sub 2}
\end{equation}%
where the dot denotes partial derivative with respect to $y^{b}.$ The
operator $\nabla _{b}$ is the covariant differentiation on $M$ with respect
to $\Gamma _{ab}^{c}$ and can be extended to tensor field on $M$ of mixed
type. For example%
\begin{equation*}
\nabla _{c}\nabla _{b}B_{a}^{r}=\partial _{c}\left( \nabla
_{b}B_{a}^{r}\right) +\Gamma _{st}^{r}B_{c}^{s}\nabla _{b}B_{a}^{t}-\Gamma
_{cb}^{d}\nabla _{d}B_{a}^{r}-\Gamma _{ca}^{d}\nabla _{b}B_{d}^{r}.
\end{equation*}

\QTP{Body Math}
For any fixed indices $a$ and $b,$ the vector $\nabla _{b}B_{a}^{r}\partial
_{r}$ is orthogonal to the submanifold. Hence%
\begin{equation*}
\nabla _{b}B_{a}^{r}\partial _{r}=h_{ab}^{x}N_{x}^{r}\partial _{r},
\end{equation*}%
where $N_{x}^{r}\partial _{r},$ $x=m+1,...,n$ are unit vectors normal to the
submanifold. For fixed $x,$ the $h_{ab}^{x}$ are components of the symmetric 
$(0,2)$ tensor $h$ on $M,$ called the second fundamental form. Consequently,
we have the decomposition%
\begin{equation*}
\nabla _{c}\nabla _{b}B_{a}^{r}=\nabla
_{c}h_{ba}^{x}N_{x}^{r}+h_{ba}^{x}\nabla _{c}N_{x}^{r},
\end{equation*}%
where $\left( \nabla _{c}N_{x}^{r}\right) \partial _{r}$ is tangent to the
submanifold for all $c$ and $x.$

\QTP{Body Math}
The Gauss formula is%
\begin{equation*}
\widetilde{\nabla }_{X}Y=\nabla _{X}Y+h(X,Y),
\end{equation*}%
for all vector fields \thinspace $X,$ $Y$ tangent to $M,$ where $\widetilde{%
\nabla }$ is the Levi-Civita connection on $N$.

\QTP{Body Math}
The Weingarten formula is 
\begin{equation*}
\widetilde{\nabla }_{X}\eta =-\widetilde{A}_{\eta }X+\widetilde{D}_{\eta }X,
\end{equation*}%
where $X$ is a tangent vector filed and $\eta $ is a normal one. $\widetilde{%
A}$ is called the shape operator while $\widetilde{D}$ is Levi-Civita
connection induced in the normal bundle over $M.$ We have%
\begin{equation*}
g(\widetilde{A}_{\eta }X,Y)=g(h(X,Y),\eta ).
\end{equation*}

\QTP{Body Math}
A submanifold $M$ is said to be totally geodesic if the second fundamental
form $h$ vanishes identically, equivalently, if the shape operator $%
\widetilde{A}$ vanishes identically. For more details see (\cite{Yano}) or (%
\cite{Yano Kon})$.$

\section{Preliminaries on $g-$ natural metrics}

In (\cite{5}) the class of $g-$natural metrics was defined. We have

\begin{lemma}
(\cite{5},\cite{6}, \cite{7}) Let $(M,g)$ be a Riemannian manifold and $G$
be a $g-$ natural metric on $TM.$ There exist functions $a_{j},$ $%
b_{j}:<0,\infty )\longrightarrow R,$ $j=1,2,3,$ such that for every $X,$ $Y,$
$u\in T_{x}M$%
\begin{multline*}
G_{(x,u)}(X^{h},Y^{h})=(a_{1}+a_{3})(r^{2})g_{x}(X,Y)+(b_{1}+b_{3})(r^{2})g_{x}(X,u)g_{x}(Y,u),
\\
G_{(x,u)}(X^{h},Y^{v})=G_{(x,u)}(X^{v},Y^{h})=a_{2}(r^{2})g_{x}(X,Y)+b_{2}(r^{2})g_{x}(X,u)g_{x}(Y,u),
\\
G_{(x,u)}(X^{v},Y^{v})=a_{1}(r^{2})g_{x}(X,Y)+b_{1}(r^{2})g_{x}(X,u)g_{x}(Y,u),
\end{multline*}%
where $r^{2}=g_{x}(u,u).$ For $\dim M=1$ the same holds for $b_{j}=0,$ $%
j=1,2,3.$
\end{lemma}

Setting $a_{1}=1,$ $a_{2}=a_{3}=b_{j}=0$ we obtain the Sasaki metric, while
setting $a_{1}=b_{1}=\frac{1}{1+r^{2}},$ $a_{2}=b_{2}=0=0,$ $a_{1}+a_{3}=1,$ 
$b_{1}+b_{3}=1$ we get the Cheeger-Gromoll one.

Following (\cite{6}) we put

\begin{enumerate}
\item $a(t)=a_{1}(t)\left( a_{1}(t)+a_{3}(t)\right) -a_{2}^{2}(t),$

\item $F_{j}(t)=a_{j}(t)+tb_{j}(t),$

\item $F(t)=F_{1}(t)\left[ F_{1}(t)+F_{3}(t)\right] -F_{2}^{2}(t)$

for all $t\in <0,\infty ).$
\end{enumerate}

We shall often abbreviate: $A=a_{1}+a_{3},$ $B=b_{1}+b_{3}.$

\begin{lemma}
\label{Lemma 9}(\cite{6}, Proposition 2.7) The necessary and sufficient
conditions for a $g-$ natural metric $G$ on the tangent bundle of a
Riemannian manifold $(M,g)$ to be non-degenerate are $a(t)\neq 0$ and $%
F(t)\neq 0$ for all $t\in <0,\infty ).$ If $\dim M=1$ this is equivalent to $%
a(t)\neq 0$ for all $t\in <0,\infty ).$
\end{lemma}

We also have

\begin{proposition}
(\cite{7-1}, \cite{7-2})Let $(N,g)$ be a Riemannian manifold, $\nabla $ its
Levi-Civita connection and $R$ its Riemann curvature tensor. If $G$ is a
non-degenerate $g-$ natural metric on $TN,$ then the Levi-Civita connection $%
\widetilde{\nabla }$ of $(TN,G)$ at a point $(x,u)\in TN$ is given by%
\begin{equation*}
\left( \widetilde{\nabla }_{X^{h}}Y^{h}\right) _{(x,u)}=\left( \nabla
_{X}Y\right) _{(x,u)}^{h}+h\left\{ \mathbf{A}(u,X_{x},Y_{x})\right\}
+v\left\{ \mathbf{B}(u,X_{x},Y_{x})\right\} ,
\end{equation*}%
\begin{equation*}
\left( \widetilde{\nabla }_{X^{h}}Y^{v}\right) _{(x,u)}=\left( \nabla
_{X}Y\right) _{(x,u)}^{v}+h\left\{ \mathbf{C}(u,X_{x},Y_{x})\right\}
+v\left\{ \mathbf{D}(u,X_{x},Y_{x})\right\} ,
\end{equation*}%
\begin{equation*}
\left( \widetilde{\nabla }_{X^{v}}Y^{h}\right) _{(x,u)}=h\left\{ \mathbf{C}%
(u,Y_{x},X_{x})\right\} +v\left\{ \mathbf{D}(u,Y_{x},X_{x})\right\} ,
\end{equation*}%
\begin{equation*}
\left( \widetilde{\nabla }_{X^{v}}Y^{v}\right) _{(x,u)}=h\left\{ \mathbf{E}%
(u,X_{x},Y_{x})\right\} +v\left\{ \mathbf{F}(u,X_{x},Y_{x})\right\} ,
\end{equation*}%
where $\mathbf{A},$ $\mathbf{B},$ $\mathbf{C},$ $\mathbf{D},$ $\mathbf{E},$ $%
\mathbf{F}$ are some F-tensors defined on the product $TN\otimes TN\otimes
TN.$
\end{proposition}

\begin{remark}
Expressions for $\mathbf{A},$ $\mathbf{B},$ $\mathbf{C},$ $\mathbf{D},$ $%
\mathbf{E},$ $\mathbf{F}$ were presented for the first time in the original
papers (\cite{6},\cite{7}). Unfortunately, they contain some misprints and
omissions. Therefore, for the correct form, we reefer the rider to ((\cite%
{7-1}, \cite{7-2}), see also (\cite{8} \cite{11}).
\end{remark}

\section{Lifts of the vectors fields}

Let $f:M\longrightarrow N$ be an isometric immersion of a Riemannian
manifold $M$ into a Riemannian manifold $N.$ Suppose that the following
diagram holds 
\begin{equation*}
\begin{array}{ccccc}
(\pi ^{-1}(U),(x^{r},u^{r}))~~TN & \boldsymbol{\longleftarrow -} & 
\widetilde{f} & \boldsymbol{--} & (\pi ^{-1}(V),(y^{a},v^{a}))~~TM \\ 
\begin{array}{c}
\boldsymbol{|} \\ 
\boldsymbol{|}%
\end{array}
&  &  &  & 
\begin{array}{c}
\boldsymbol{|} \\ 
\boldsymbol{|}%
\end{array}
\\ 
\pi _{N} &  &  &  & \pi _{M} \\ 
\begin{array}{c}
\boldsymbol{|} \\ 
\boldsymbol{\downarrow }%
\end{array}
&  &  &  & 
\begin{array}{c}
\boldsymbol{|} \\ 
\boldsymbol{\downarrow }%
\end{array}
\\ 
(U,(x^{r}))~~N & \boldsymbol{\longleftarrow -} & f & \boldsymbol{--} & 
(V,(y^{a}))~~M%
\end{array}%
,
\end{equation*}
where $(U,(x^{r}))$ and $(V,(y^{a}))$ are coordinate neighbourhoods on $N$
and $M$ respectively while the local expression for $f$ is: $%
x^{r}=x^{r}(y^{a}),$ $r=1,...,n,$ $a=1,...,m$ and $m<n.$ Then the local
coordinate vector fields on $M$ are given by $\frac{\delta }{\delta y^{a}}=%
\frac{\partial x^{r}}{\partial y^{a}}\frac{\partial }{\partial x^{r}}%
=B_{a}^{r}\frac{\partial }{\partial x^{r}}.$

Define the map 
\begin{equation}
\widetilde{f}:x^{r}=x^{r}(y^{a}),\ u^{r}=v^{a}B_{a}^{r}.  \label{Lift 2}
\end{equation}
This is an immersion of the rank $2m$ since the Jacobi matrix $J$ is of the
form 
\begin{equation*}
J=%
\begin{bmatrix}
\frac{\partial x^{r}}{\partial y^{a}} & \frac{\partial x^{r}}{\partial v^{a}}
\\ 
\frac{\partial u^{r}}{\partial y^{a}} & \frac{\partial u^{r}}{\partial v^{a}}%
\end{bmatrix}%
=%
\begin{bmatrix}
B_{a}^{r} & 0 \\ 
v^{b}\partial _{b}B_{a}^{r} & B_{a}^{r}%
\end{bmatrix}%
.
\end{equation*}%
Since $\frac{\delta }{\delta y^{a}}=B_{a}^{r}\frac{\partial }{\partial x^{r}}
$ the vectors tangent to $TM$ are%
\begin{equation*}
\frac{\partial }{\partial y^{a}}=\frac{\delta }{\delta y^{a}}%
+v^{b}B_{a.b}^{r}\frac{\partial }{\partial u^{r}},\quad \frac{\partial }{%
\partial v^{a}}=B_{a}^{r}\frac{\partial }{\partial u^{r}}.
\end{equation*}

\begin{definition}
\label{Lift of the immersion}Let $f:M\longrightarrow N$ be an isometric
immersion. Then the map $\widetilde{f}:TM\longrightarrow TN,$ locally given
by $x^{r}=x^{r}(y^{a}),\ u^{r}=v^{a}B_{a}^{r},$ will be called the lift of
the immersion $f$ and its image $LM$ will be called the lift of the
submanifold $M.$
\end{definition}

\begin{remark}
The lift of an immersion defined above seems to be the most natural since $%
B_{a}^{r}\partial _{r}|_{(x^{r}(y^{a}))}$ are coordinate vectors at the
point $(x^{r}(y^{a}))\in M$ and $\left( v^{a}\right) $ are components of
tangent vectors. This kind of lift appears quite natural (c.f. \cite{3-1}).
\end{remark}

\begin{lemma}
\label{Lifts of coordinate vf}Let $\widetilde{f}:TM\longrightarrow TN$
locally given by $x^{r}=x^{r}(y^{a}),\ u^{r}=v^{a}B_{a}^{r}$ be the lift of
the isometric immersion $f:M\longrightarrow N.$ Then the vertical and
horizontal lifts of the coordinate vector field $\frac{\delta }{\delta y^{a}}
$ on $M$ with respect to $TN$ an $TM$ are related by%
\begin{equation*}
\left( \frac{\delta }{\delta y^{a}}\right) ^{v_{N}}=\left( \frac{\delta }{%
\delta y^{a}}\right) ^{v_{M}},
\end{equation*}%
\begin{equation*}
\left( \frac{\delta }{\delta y^{a}}\right) ^{h_{N}}=\left( \frac{\delta }{%
\delta y^{a}}\right) ^{h_{M}}-v^{b}\nabla _{b}B_{a}^{r}\left( \frac{\partial 
}{\partial x^{r}}\right) ^{v_{N}}.
\end{equation*}
\end{lemma}

\begin{proof}
By the use of the definitions of \ $v_{N}$ and $v_{M}$ we have 
\begin{equation*}
\left( \frac{\delta }{\delta y^{a}}\right) ^{v_{N}}=B_{a}^{r}\left( \frac{%
\partial }{\partial x^{r}}\right) ^{v_{N}}=B_{a}^{r}\frac{\partial }{%
\partial u^{r}}=\frac{\partial }{\partial v^{a}}=\left( \frac{\delta }{%
\delta y^{a}}\right) ^{v_{M}}.
\end{equation*}

For horizontal lifts we have%
\begin{equation*}
\left( \frac{\delta }{\delta y^{a}}\right) ^{h_{M}}=\frac{\partial }{%
\partial y^{a}}-\Gamma _{ab}^{c}v^{b}\frac{\partial }{\partial v^{c}}
\end{equation*}%
and, along $M,$%
\begin{multline*}
\left( \frac{\delta }{\delta y^{a}}\right) ^{h_{N}}=B_{a}^{r}\left( \frac{%
\partial }{\partial x^{r}}\right) ^{h_{N}}=B_{a}^{r}\left( \frac{\partial }{%
\partial x^{r}}-u^{s}\Gamma _{sr}^{t}\frac{\partial }{\partial u^{t}}\right)
= \\
\frac{\delta }{\delta y^{a}}-v^{b}B_{b}^{r}B_{a}^{s}\Gamma _{rs}^{t}\frac{%
\partial }{\partial u^{t}}\overset{(1)}{=}\frac{\delta }{\delta y^{a}}-v^{b}%
\left[ \nabla _{b}B_{a}^{r}-B_{a.b}^{r}+\Gamma _{ab}^{c}B_{c}^{r}\right] 
\frac{\partial }{\partial u^{r}}= \\
\left( \frac{\delta }{\delta y^{a}}+v^{b}B_{a.b}^{r}\frac{\partial }{%
\partial u^{r}}\right) -v^{b}\Gamma _{ab}^{c}\frac{\partial }{\partial v^{c}}%
-v^{b}\nabla _{b}B_{a}^{r}\frac{\partial }{\partial u^{r}}= \\
\left( \frac{\delta }{\delta y^{a}}\right) ^{h_{M}}-v^{b}\nabla _{b}B_{a}^{r}%
\frac{\partial }{\partial u^{r}}=\left( \frac{\delta }{\delta y^{a}}\right)
^{h_{M}}-v^{b}\nabla _{b}B_{a}^{r}\left( \frac{\partial }{\partial x^{r}}%
\right) ^{v_{N}}.
\end{multline*}
\end{proof}

We define the vertical vector field%
\begin{equation*}
K_{a}=v^{b}\nabla _{b}B_{a}^{r}\left( \frac{\partial }{\partial x^{r}}%
\right) ^{v_{N}}=K_{a}^{r}\left( \frac{\partial }{\partial x^{r}}\right)
^{v_{N}}.
\end{equation*}

\begin{corollary}
If $M$ is a totally geodesic submanifolds in $N,$ then the horizontal lifts $%
h_{N}$ and $h_{M}$ coincide on $TM$.
\end{corollary}

\section{Projections}

\begin{lemma}
Let $f:M\longrightarrow N$ be an isometric immersion of Riemannian manifolds
and $\widetilde{f}:TM\longrightarrow TN$ be its lift defined by (\ref{Lift 2}%
). Then the projections $\pi _{N}:TN\longrightarrow N$ and $\pi
_{M}:TM\longrightarrow M$ satisfy%
\begin{equation*}
d\left( \pi _{N}\right) |_{TTM}=d\left( \pi _{M}\right) .
\end{equation*}
\end{lemma}

\begin{proof}
The components of the projections: $\pi _{N}:TN\longrightarrow N$ and $\pi
_{M}:TM\longrightarrow M$ can be written as%
\begin{equation*}
\left( \pi _{N}\right) ^{s}(x^{r},u^{r})=x^{s},\quad \left( \pi _{M}\right)
^{b}(y^{a},v^{a})=y^{b},
\end{equation*}
where $r=1,...,n,$ $a=1,...,m.$ Thus for an arbitrary fixed\ indices $s,$ $b$
we obtain%
\begin{equation*}
\frac{\partial \left( \pi _{N}\right) ^{r}}{\partial x^{t}}=\delta
_{t}^{r},\quad \frac{\partial \left( \pi _{N}\right) ^{r}}{\partial u^{t}}%
=0,\quad \frac{\partial \left( \pi _{M}\right) ^{a}}{\partial y^{b}}=\delta
_{b}^{a},\quad \frac{\partial \left( \pi _{M}\right) ^{a}}{\partial v^{b}}=0.
\end{equation*}%
Then for vectors $\frac{\delta }{\delta y^{c}},$ $c=1,...,m,$ tangent to $M,$
and $\frac{\partial }{\partial y^{c}}$ tangent to $TM$ we find%
\begin{equation*}
d\pi _{M}|_{(x^{r}(y^{a}),v^{b}B_{b}^{r})}\left( \frac{\partial }{\partial
y^{c}}\right) =\frac{\partial \left( \pi _{M}\right) ^{a}}{\partial y^{c}}%
\frac{\delta }{\delta y^{a}}=\delta _{c}^{a}\frac{\delta }{\delta y^{a}}%
|_{(x^{r}(y^{a}))}=\frac{\delta }{\delta y^{c}}|_{(x^{r}(y^{a}))},
\end{equation*}%
and%
\begin{multline*}
d\pi _{N}|_{(x^{r}(y^{a}),v^{b}B_{b}^{r})}\left( \frac{\partial }{\partial
y^{c}}\right) =d\pi _{N}|_{(x^{r}(y^{a}),v^{b}B_{b}^{r})}\left( B_{c}^{r}%
\frac{\partial }{\partial x^{r}}+v^{b}B_{b}^{r}\frac{\partial }{\partial
u^{r}}\right) = \\
B_{c}^{r}\left[ d\pi _{N}|_{(x^{r}(y^{a}),v^{b}B_{b}^{r})}\left( \frac{%
\partial }{\partial x^{r}}\right) \right] +v^{b}B_{b}^{r}\left[ d\pi
_{N}|_{(x^{r}(y^{a}),v^{b}B_{b}^{r})}\left( \frac{\partial }{\partial u^{r}}%
\right) \right] = \\
B_{c}^{r}\left( \frac{\partial \left( \pi _{N}\right) ^{s}}{\partial x^{r}}%
\frac{\partial }{\partial x^{s}}\right)
|_{(x^{r}(y^{a}))}+v^{b}B_{b}^{r}\left( \frac{\partial \left( \pi
_{N}\right) ^{s}}{\partial u^{r}}\frac{\partial }{\partial x^{s}}\right)
|_{(x^{r}(y^{a}))}= \\
B_{c}^{r}\left( \delta _{r}^{s}\frac{\partial }{\partial x^{s}}\right)
|_{(x^{r}(y^{a}))}=B_{c}^{r}\frac{\partial }{\partial x^{s}}%
|_{(x^{r}(y^{a}))}=\frac{\delta }{\delta y^{c}}|_{(x^{r}(y^{a}))}.
\end{multline*}%
Similarly, for $\frac{\partial }{\partial v^{c}}$ tangent to $TM,$ we get%
\begin{equation*}
d\pi _{M}|_{(x^{r}(y^{a}),v^{b}B_{b}^{r})}\left( \frac{\partial }{\partial
v^{c}}\right) =\frac{\partial \left( \pi _{M}\right) ^{a}}{\partial v^{c}}%
\frac{\delta }{\delta y^{a}}=0,
\end{equation*}%
and%
\begin{multline*}
d\pi _{N}|_{(x^{r}(y^{a}),v^{b}B_{b}^{r})}\left( \frac{\partial }{\partial
v^{c}}\right) =d\pi _{N}|_{(x^{r}(y^{a}),v^{b}B_{b}^{r})}\left( B_{c}^{r}%
\frac{\partial }{\partial u^{r}}\right) = \\
B_{c}^{r}\left[ d\pi _{N}|_{(x^{r}(y^{a}),v^{b}B_{b}^{r})}\left( \frac{%
\partial }{\partial u^{r}}\right) \right] =B_{c}^{r}\left( \frac{\partial
\left( \pi _{N}\right) ^{s}}{\partial u^{r}}\frac{\partial }{\partial x^{s}}%
\right) |_{(x^{r}(y^{a}))}=0.
\end{multline*}%
This completes the proof.
\end{proof}

\section{Connection map}

\begin{lemma}
Let $\widetilde{f}:TM\longrightarrow TN$ be the lift of the immersion $%
f:M\longrightarrow N$ in the sense of Definition \ref{Lift of the immersion}%
. Then the connection maps $K_{N}$ and $K_{M}$ with respect to the
connections $\nabla _{N}$ and $\nabla _{M}$ respectively satisfy%
\begin{equation*}
K_{N}(\frac{\partial }{\partial v^{a}})=\frac{\delta }{\delta y^{a}}=K_{M}(%
\frac{\partial }{\partial v^{a}}),
\end{equation*}%
\begin{equation*}
K_{N}(\frac{\partial }{\partial y^{a}})=v^{b}\nabla _{b}B_{a}^{r}\frac{%
\partial }{\partial x^{r}}+K_{M}(\frac{\partial }{\partial y^{a}}).
\end{equation*}
\end{lemma}

\begin{proof}
By definition we have%
\begin{equation*}
K_{N}:T_{(x,u)}TN\longrightarrow T_{x}N,\quad K_{N}(X^{r}\frac{\partial }{%
\partial x^{r}}+\overline{X}^{r}\frac{\partial }{\partial u^{r}})=\left( 
\overline{X}^{r}+\Gamma _{st}^{r}u^{s}X^{t}\right) \frac{\partial }{\partial
x^{r}},
\end{equation*}%
\begin{equation*}
K_{M}:T_{(y,v)}TM\longrightarrow T_{y}M,\quad K_{M}(Z^{a}\frac{\partial }{%
\partial y^{a}}+\overline{Z}^{a}\frac{\partial }{\partial v^{a}})=\left( 
\overline{Z}^{a}+\Gamma _{bc}^{a}v^{b}Z^{c}\right) \frac{\delta }{\delta
y^{a}}.
\end{equation*}%
Hence, we obtain%
\begin{equation*}
K_{N}(\frac{\partial }{\partial v^{a}})=K_{N}(B_{a}^{r}\frac{\partial }{%
\partial u^{r}})=B_{a}^{r}\frac{\partial }{\partial x^{r}}=\frac{\delta }{%
\delta y^{a}}=K_{M}(\frac{\partial }{\partial v^{a}}).
\end{equation*}%
Moreover,%
\begin{equation*}
K_{M}(\frac{\partial }{\partial y^{a}})=\Gamma _{db}^{c}v^{b}\delta _{a}^{d}%
\frac{\delta }{\delta y^{c}}=\Gamma _{ab}^{c}v^{b}B_{c}^{t}\frac{\partial }{%
\partial x^{t}},
\end{equation*}%
and, by the use of (\ref{Sub 2}) and (\ref{Lift 2}), we find 
\begin{multline*}
K_{N}(\frac{\partial }{\partial y^{a}})=K_{N}(B_{a}^{r}\frac{\partial }{%
\partial x^{r}}+v^{b}\partial _{b}B_{a}^{r}\frac{\partial }{\partial u^{r}}%
)=\left( v^{b}\partial _{b}B_{a}^{r}+\Gamma _{st}^{r}B_{a}^{s}u^{t}\right) 
\frac{\partial }{\partial x^{r}}= \\
v^{b}\left( \partial _{b}B_{a}^{r}+\Gamma _{st}^{r}B_{a}^{s}B_{b}^{t}\right) 
\frac{\partial }{\partial x^{r}}=v^{b}\left( \nabla _{b}B_{a}^{r}+\Gamma
_{ab}^{c}B_{c}^{r}\right) \frac{\partial }{\partial x^{r}}= \\
v^{b}\nabla _{b}B_{a}^{r}\frac{\partial }{\partial x^{r}}+v^{b}\Gamma
_{ab}^{c}\frac{\delta }{\delta y^{c}}=v^{b}\nabla _{b}B_{a}^{r}\frac{%
\partial }{\partial x^{r}}+K_{M}(\frac{\partial }{\partial y^{a}}).
\end{multline*}%
Thus the Lemma is proved.
\end{proof}

\section{Vector field normal to $LM$}

In the case of $M$ being totally geodesic in $N,$ the unit vector fields
normal to the lift of $M$ can be chosen in the form $\alpha \eta
_{x}^{h_{N}}+\beta \eta _{x}^{h_{N}},$ $x=m+1,...,n,$ while $\alpha ,$ $%
\beta $ are functions depending on generators of the $g-$ natural metric $G$
along the lift. The next lemma explains the structure of the vector field
normal to $LM$ in the general case.

\begin{lemma}
Suppose that $M\ $is not necessary totally geodesic in $N\ $and $\ \eta
=H_{\top }^{h_{N}}+H_{\bot }^{h_{N}}+V_{\top }^{v_{N}}+V_{\bot }^{v_{N}}$ is
a vector field normal to the lifted submanifold $LM$, where $H_{\top },$ $%
V_{\top }$ are tangent to $M$ and $H_{\bot },$ $V_{\bot }$ are normal to $M$
in $TN.$ Then for all $u$ tangent to $M\ $we have%
\begin{equation*}
g(H_{\top },u)=-\frac{F_{1}}{F}g(K,\ a_{2}H_{\bot }+a_{1}V_{\bot }),
\end{equation*}%
\begin{equation*}
g(V_{\top },u)=\frac{F_{2}}{F}g(K,\ a_{2}H_{\bot }+a_{1}V_{\bot }),
\end{equation*}%
where $K=v^{c}K_{c}=v^{c}v^{a}\nabla _{c}B_{a}^{r}\partial _{r}.$ Moreover,
if $a_{1}=const,$ $a_{2}=const,$ then%
\begin{equation*}
g(H_{\top },\delta _{a})=-\frac{F_{1}}{F}g(K_{a},\ a_{2}H_{\bot
}+a_{1}V_{\bot }),
\end{equation*}%
\begin{equation*}
g(V_{\top },\delta _{a})=\frac{F_{2}}{F}g(K_{a},\ a_{2}H_{\bot
}+a_{1}V_{\bot }).
\end{equation*}
\end{lemma}

\begin{proof}
Relation $G(\delta _{a}^{v_{N}},\eta )=0$ yields%
\begin{equation*}
a_{2}H_{\top }+b_{2}g(H_{\top },u)u+a_{1}V_{\top }+b_{1}g(V_{\top },u)u=0,
\end{equation*}%
whence, by contraction with $u,$ we get 
\begin{equation}
F_{2}g(H_{\top },u)+F_{1}g(V_{\top },u)=0.  \label{Vfn to LM 2}
\end{equation}

On the other hand, relations $G(\delta _{a}^{h_{M}},\eta )=0$ and $%
h_{M}=h_{N}+v_{N}$ yield%
\begin{equation}
g\left( \delta _{a},\ AH_{\top }+Bg(H_{\top },u)u+a_{2}V_{\top
}+b_{2}g(V_{\top },u)u\right) +g(K_{a},\ a_{2}H_{\bot }+a_{1}V_{\bot })=0,
\label{Vfn to LM 3}
\end{equation}%
where $K_{a}=v^{c}\nabla _{c}B_{a}^{r}\partial _{r}.$ Transvecting (\ref{Vfn
to LM 3}) with $v^{a},$ we obtain%
\begin{equation}
\left( F_{1}+F_{3}\right) g(H_{\top },u)+F_{2}g(V_{\top },u)=-g(K,\
a_{2}H_{\bot }+a_{1}V_{\bot }),  \label{Vfn to LM 4}
\end{equation}%
where $K=v^{c}K_{c}.$

Since $G\ $is non-degenerate, $F=F_{1}(F_{1}+F_{3})-F_{2}^{2}\neq 0.$
Solving the system consisting of (\ref{Vfn to LM 2}) and (\ref{Vfn to LM 4})
with respect to $g(H_{\top },u)\ $and $g(H_{\top },u)$ we obtain the thesis.
\end{proof}

Consequently, if $M$ is totally geodesic in $N,$ then necessary $H_{\top
}=V_{\top }=0.$

\section{Lift of a totally geodesic submanifold}

\subsection{Normal bundle}

Suppose that $M$ is a totally geodesic submanifold isometricaly immersed in $%
N$ and $\eta _{x},$ $x=m+1,...,n$ are vector fields normal to $M$ in $TN.$
Then, by the Lemma \ref{Lifts of coordinate vf}, the lifts of the vector
fields from $M$ to $TM$ coincide with those to $TN.$ The lifts $\left( \eta
_{x}\right) ^{h}=\left( \eta _{x}\right) ^{h_{N}}$ and $\left( \eta
_{x}\right) ^{v}=\left( \eta _{x}\right) ^{v_{N}}$ are orthogonal to $\left( 
\frac{\delta }{\delta y^{a}}\right) ^{h_{M}}$ and $\left( \frac{\delta }{%
\delta y^{b}}\right) ^{v_{M}}$ but are orthogonal to each other if and only
if $a_{2}=0$ since $G\left( \left( \eta _{x}\right) ^{h},\left( \eta
_{y}\right) ^{v}\right) =a_{2}g(\eta _{x},\eta _{y})$ for all $%
x,y=m+1,...,n. $

\begin{proposition}
\label{Normal vf - totg case}Let $M,$ $\dim M=m,$ be a totally geodesic
submanifold isometrically immersed in a Riemannian manifold $(N,g),$ $\dim
N=n$ and $\eta _{x},$ $x=m+1,...,n$ be a set of vector fields normal to $M$
in $TN.$ Suppose, moreover, that $TN$ is endowed with $g-$ natural metric $G$
and $a=a_{1}A-a_{2}^{2}\neq 0.$

If $LM$ denotes the lift (\ref{Lift 2}) of $M$ to $TN,$ then the normal
bundle of $LM$ in $TTN$ is spanned by vector fields $S_{x},$ $T_{x},$ $%
x=m+1,...,n$ and the following six cases occur:

\begin{enumerate}
\item $Aa_{1}\neq 0,$ $a_{2}$ arbitrary. Let $\varepsilon
=sgn(a_{1}A-a_{2}^{2}),$ $\delta =sgna_{1}$ and%
\begin{equation*}
S_{x}=\frac{\varepsilon \delta \sqrt{\left\vert a_{1}\right\vert }}{\sqrt{%
\left\vert a\right\vert }}\left( \eta _{x}\right) ^{h}-\frac{\varepsilon
\delta a_{2}\sqrt{\left\vert a_{1}\right\vert }}{a_{1}\sqrt{\left\vert
a\right\vert }}\left( \eta _{x}\right) ^{v},\text{ }T_{x}=\frac{\delta }{%
\sqrt{\left\vert a_{1}\right\vert }}\left( \eta _{x}\right) ^{v}.
\end{equation*}%
Then%
\begin{equation*}
G(S_{x},S_{y})=\varepsilon \delta g(\eta _{x},\eta _{y}),\quad
G(S_{x},T_{y})=0,\quad G(T_{x},T_{y})=\delta g(\eta _{x},\eta _{y}).
\end{equation*}

\item $a_{2}\neq 0,$ $A=0,$ $a_{1}=0,\varepsilon =sgna_{2}$ and%
\begin{equation*}
S_{x}=\frac{\varepsilon }{2a_{2}}\eta _{x}^{h}-\eta _{x}^{v},\text{ }T_{x}=%
\frac{\varepsilon }{2a_{2}}\eta _{x}^{h}+\eta _{x}^{v}.
\end{equation*}%
Then%
\begin{equation*}
G(S_{x},S_{y})=-\varepsilon g(\eta _{x},\eta _{y}),\quad
G(S_{x},T_{y})=0,\quad G(T_{x},T_{y})=\varepsilon g(\eta _{x},\eta _{y}).
\end{equation*}

\item $a_{2}\neq 0,$ $A=0,$ $a_{1}\neq 0,$ $\varepsilon =sgna_{1}=-1$ and%
\begin{equation*}
S_{x}=\frac{a_{1}}{\sqrt{3}\sqrt{\left\vert a_{1}\right\vert }a_{2}}\eta
_{x}^{h}+\frac{1}{\sqrt{3}\sqrt{\left\vert a_{1}\right\vert }}\eta
_{x}^{v},\quad T_{x}=\frac{2a_{1}}{\sqrt{3}\sqrt{\left\vert a_{1}\right\vert 
}a_{2}}\eta _{x}^{h}-\frac{1}{\sqrt{3}\sqrt{\left\vert a_{1}\right\vert }}%
\eta _{x}^{v}.
\end{equation*}%
Then%
\begin{equation*}
G(S_{x},S_{y})=-g(\eta _{x},\eta _{y}),\quad G(S_{x},T_{y})=0,\quad
G(T_{x},T_{y})=g(\eta _{x},\eta _{y}).
\end{equation*}

\item $a_{2}\neq 0,$ $A=0,$ $a_{1}\neq 0,$ $\varepsilon =sgna_{1}=1$ and%
\begin{equation*}
S_{x}=\frac{a_{1}}{\sqrt{a_{1}}a_{2}}\eta _{x}^{h}-\frac{1}{\sqrt{a_{1}}}%
\eta _{x}^{v},\quad T_{x}=\frac{1}{\sqrt{a_{1}}}\eta _{x}^{v}.
\end{equation*}%
Then%
\begin{equation*}
G(S_{x},S_{y})=-g(\eta _{x},\eta _{y}),\quad G(S_{x},T_{y})=0,\quad
G(T_{x},T_{y})=g(\eta _{x},\eta _{y}).
\end{equation*}

\item $a_{2}\neq 0,$ $A\neq 0,$ $a_{1}=0,$ $\varepsilon =sgnA=-1$ and%
\begin{equation*}
S_{x}=\frac{A}{\sqrt{3}\sqrt{\left\vert A\right\vert }a_{2}}\eta _{x}^{v}+%
\frac{1}{\sqrt{3}\sqrt{\left\vert A\right\vert }}\eta _{x}^{h},\quad T_{x}=%
\frac{2A}{\sqrt{3}\sqrt{\left\vert A\right\vert }a_{2}}\eta _{x}^{v}-\frac{1%
}{\sqrt{3}\sqrt{\left\vert A\right\vert }}\eta _{x}^{h}.
\end{equation*}%
Then%
\begin{equation*}
G(S_{x},S_{y})=-g(\eta _{x},\eta _{y}),\quad G(S_{x},T_{y})=0,\quad
G(T_{x},T_{y})=g(\eta _{x},\eta _{y}).
\end{equation*}

\item $a_{2}\neq 0,$ $A\neq 0,$ $a_{1}=0,$ $\varepsilon =sgnA=1$ and%
\begin{equation*}
S_{x}=\frac{A}{\sqrt{A}a_{2}}\eta _{x}^{v}-\frac{1}{\sqrt{A}}\eta
_{x}^{h},\quad T_{x}=\frac{1}{\sqrt{A}}\eta _{x}^{h}.
\end{equation*}%
Then%
\begin{equation*}
G(S_{x},S_{y})=-g(\eta _{x},\eta _{y}),\quad G(S_{x},T_{y})=0,\quad
G(T_{x},T_{y})=g(\eta _{x},\eta _{y}).
\end{equation*}
\end{enumerate}

If $\eta _{x}$ are unit, so do $S_{x}$ and $T_{x}.$

Moreover, only in the first case the metric induced on the normal bundle
from $G$ can be a Riemannian one.
\end{proposition}

\subsection{Submanifold LM}

We abbreviate $\delta _{a}=\frac{\delta }{\delta y^{a}}.$ Denote by $%
\widetilde{\nabla }$ the Levi-Civita connection of the $g-$ natural metric $G
$ on $TN$ and by $\nabla $ the Levi-Civita connection on $N.$ Moreover, let $%
\widetilde{A}$ denotes the shape operator of the lifted submanifold $LM.$
Since $M$ is totally geodesic, the normal bundle of the lifted submanifold
is spanned by vector fields of the form $\alpha \eta _{x}^{h_{N}}+\beta \eta
_{x}^{v_{N}},$ $x=m+1,...,n,$ $\alpha =\alpha (v^{2}),$ $\beta =\beta (v^{2})
$ being functions depending on $v^{2}=g_{ab}v^{a}v^{b}$ and $\eta _{x}$ are
vector fields on $N$ normal to $M.$ Moreover we have $\delta
_{a}^{h_{N}}=\delta _{a}^{h_{M}.}.$ Setting $\delta _{a}^{h_{N}}=\delta
_{a}^{h}$, by the use of the Weingarten formula to the normal vector field $%
\eta $, we obtain along $LM$%
\begin{multline*}
-G\left( \widetilde{A}_{\eta ^{h}}\left( \delta _{a}^{h}\right) ,\delta
_{b}^{h}\right) =G\left( \widetilde{\nabla }_{\delta _{a}^{h}}\left( \eta
^{h}\right) ,\delta _{b}^{h}\right) = \\
Ag(\nabla _{\delta _{a}}\eta ,\delta _{a})+Bg(\nabla _{\delta _{a}}\eta
,u)g(\delta _{a},u)+ \\
Ag(\mathbf{A}(u,\delta _{a},\eta ),\delta _{a})+Bg(\mathbf{B}(u,\delta
_{a},\eta ),u)g(\delta _{a},u)+ \\
a_{2}g(\mathbf{A}(u,\delta _{a},\eta ),\delta _{a})+b_{2}g(\mathbf{B}%
(u,\delta _{a},\eta ),u)g(\delta _{a},u).
\end{multline*}%
Since $M$ is totally geodesic, the two first terms on the right hand side
vanish. Computer supported computations show that%
\begin{equation*}
-G\left( \widetilde{A}_{\eta ^{h}}\left( \delta _{a}^{h}\right) ,\delta
_{b}^{h}\right) =a_{2}R(u,\delta _{a},\eta ,\delta _{b}).
\end{equation*}%
Similarly, we find%
\begin{equation*}
-G\left( \widetilde{A}_{\eta ^{h}}\left( \delta _{a}^{h}\right) ,\delta
_{b}^{v}\right) =-G\left( \widetilde{A}_{\eta ^{h}}\left( \delta
_{b}^{v}\right) ,\delta _{a}^{h}\right) =\frac{1}{2}a_{1}R(u,\delta
_{a},\eta ,\delta _{b}),
\end{equation*}

\begin{equation*}
-G\left( \widetilde{A}_{\eta ^{v}}\left( \delta _{a}^{h}\right) ,\delta
_{b}^{h}\right) =\frac{1}{2}a_{1}R(u,\eta ,\delta _{a},\delta _{b}),
\end{equation*}

\begin{multline*}
G\left( \widetilde{A}_{\eta ^{h}}\left( \delta _{a}^{v}\right) ,\delta
_{b}^{v}\right) =G\left( \widetilde{A}_{\eta ^{v}}\left( \delta
_{a}^{h}\right) ,\delta _{b}^{v}\right) = \\
G\left( \widetilde{A}_{\eta v}\left( \delta _{a}^{v}\right) ,\delta
_{b}^{h}\right) =G\left( \widetilde{A}_{\eta ^{v}}\left( \delta
_{a}^{v}\right) ,\delta _{b}^{v}\right) =0,
\end{multline*}%
where $R$ is the Riemann curvature tensor of the manifold $N.$

From the above formulas we get immediately our main result.

\begin{theorem}
If $M$ is a totally geodesic submanifold isometrically immersed in a space
of constant curvature $N,$ then the lift (\ref{Lift 2}) of $M$ to $TN$ with
non-degenerate $g-$ natural metric $G$ is totally geodesic submanifold in $%
(TN,G).$
\end{theorem}

Stanis\l aw Ewert-Krzemieniewski

West Pomeranian University of Technology Szczecin

School of Mathematics

Al. Piast\'{o}w 17, 70-310 Szczecin, Poland

e-mail: ewert@zut.edu.pl

\end{document}